\documentclass[a4paper]{article}

\usepackage[T1]{fontenc}
\usepackage[latin9]{inputenc}
\usepackage{amsmath, dsfont}
\usepackage{amssymb}
\usepackage{amsthm}
\usepackage{wasysym}
\usepackage{graphicx}
\usepackage{nicefrac}

\usepackage{fullpage}
\usepackage{enumerate}

\usepackage{enumitem}
\usepackage[labelfont=bf,width=.85\textwidth,font=small]{caption}

\usepackage{thmtools}
\usepackage{thm-restate}

\newtheorem{theorem}{Theorem}
\newtheorem{lemma}[theorem]{Lemma}

\newtheorem{observation}[theorem]{Observation}

\newtheorem{claim}[theorem]{Claim}

\newtheorem*{claim*}{Claim}
\newtheorem{step}{Step}

\newtheorem*{step*}{Step}
\theoremstyle{remark}
\newtheorem*{remark}{Remark}

\newcommand{\Prob}{Pr}
\newcommand{\A}{\ensuremath{\mathcal A}}

\newcommand{\headrule}{HeadRule}

\DeclareMathOperator{\reach}{reach}

\begin{document}
\begin{titlepage}

\begin{center}
{\huge Upper and Lower Bounds on \\ Long Dual-Paths in Line Arrangements}\\[0cm]
\end{center}

\vfill

\begin{minipage}[c]{0.3\textwidth}%
  \begin{center}
  \textsc{Udo Hoffmann}
  \\ [0.1cm]
  {\small Technische Universität Berlin}\\
  {\small Berlin, Germany \\}
  {\small \texttt{uhoffman@math.tu-berlin.de} \\[0.1cm] }
  {\tiny {Supported by the Deutsche Forschungsgemeinschaft
within the research training group 'Methods for Discrete Structures' (GRK
1408).} \\[0.4cm] }
  \par\end{center}%
\end{minipage}
\begin{minipage}[c]{0.3\textwidth}%
  \begin{center}
  \textsc{Linda Kleist}\\ [0.1cm]
  {\small Technische Universität Berlin}\\
  {\small Berlin, Germany \\ }
  {\small \texttt{kleist@math.tu-berlin.de} \\[0.4cm] }
  \par\end{center}%
\end{minipage}
\begin{minipage}[c]{0.3\textwidth}%
  \begin{center}
  \textsc{Tillmann Miltzow}\\ [0.1cm]
  {\small Freie Universit\"at Berlin \\ }
  {\small Berlin, Germany \\ }
  {\small \texttt{t.miltzow@gmail.com} \\[0.1cm] }
  {\tiny{supported by the ERC grant PARAMTIGHT: Parameterized complexity and
the search
for tight complexity results", no. 280152.} \\[0.4cm] }
  
  \par\end{center}%
\end{minipage}

\vfill

\begin{figure}[htbp]
  \centering
  \includegraphics[width = .8\textwidth]{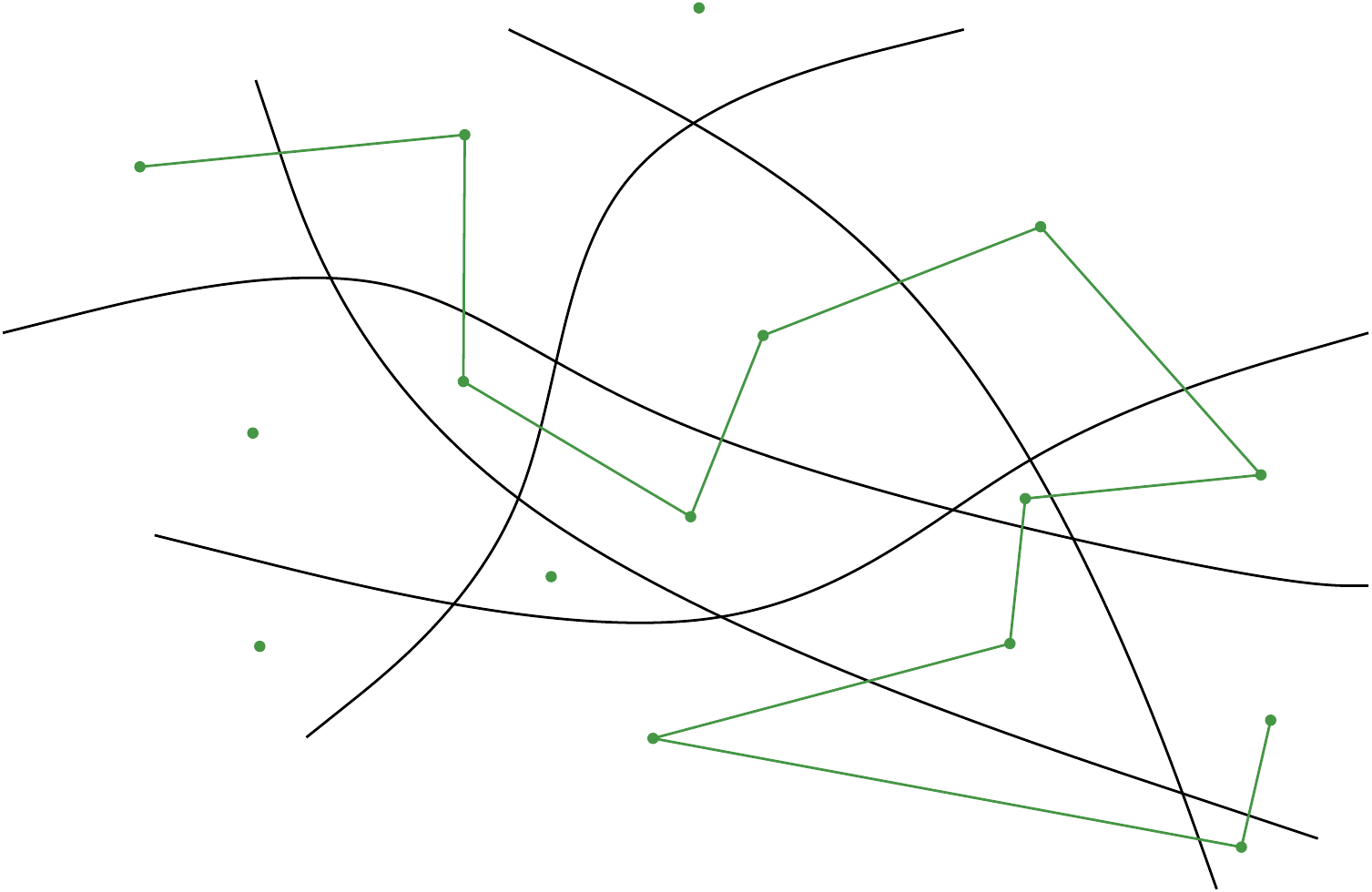}\\[1cm]
\end{figure}

\vfill

\begin{abstract}
Given a line arrangement $\A$ with $n$ lines, we show that there
exists a path of length $n^2/3 - O(n)$ in the dual graph of \A \ formed by its faces.
This bound is tight up to lower order terms.
For the bicolored version, we describe an example of a line
arrangement with $3k$ blue and $2k$ red lines with no alternating
path longer than $14k$.
Further, we show that any line arrangement with $n$ lines has a
coloring such that it has an alternating path of length $\Omega (n^2/ \log n)$.
Our results also hold for pseudoline arrangements.
\end{abstract}

\vfill

\end{titlepage}

\section{Introduction}

A \emph{line arrangement} $\mathcal A$ is a set of lines in the Euclidean plane. The set of lines partitions the plane into faces, edges and vertices and thus defines a plane graph $G(\mathcal A)$.
A line arrangement is called \emph{simple} if no two lines are parallel and every point of the Euclidean plane is covered by at most 2 lines. That is, every two lines intersect exactly once and every vertex of $G(\mathcal A)$ has degree 4.
A \emph{(dual) path} in a line arrangement is a sequence of faces such that
consecutive faces share an edge and no face appears more than once.
Alternatively, such a path can be seen as a simple path in the dual graph $G^*(\mathcal A)$ of the plane graph $G(\mathcal A)$, where the faces of the arrangement are the vertices and two faces of the arrangement are adjacent when they share an edge. 

We are interested in the length of a longest path of a simple arrangement.
In particular, we want to bound the length of the longest path in the set of arrangements with $n$ lines: Given the set $\mathbf A_n$ of arrangements with $n$ lines, we are interested in bounds on the function 
$$\displaystyle f(n):=\min_{\mathcal A\in\mathbf A_n} \max_{P\in \mathcal A} {|P|}$$
where $P$ is a dual path in $\mathcal A$ and $|P|$ its length.\\

This question has a colored counterpart. If we color the lines of a line arrangement with either red or blue (such that each color appears at least once), we talk about a \emph{bicolored line arrangement}. A dual path in a bicolored arrangement is \emph{alternating} if no two consecutive edges which certify the contact of the faces are of the same color.
For the bicolored version, we are interested in the length of a longest alternating path of a simple arrangement. We are particularly interested in the case when the number of red lines and the number of blue lines are roughly the same.

\subsection{State of the Art}

Aichholzer \emph{et al.}~\cite{cellpaths} introduced the problems.  To the best of our knowledge, no one else studied these questions previously.

Let us briefly repeat the current state of art, that is results in~\cite{cellpaths}:  
For the uncolored case, they show that any line arrangement has a dual path of
length $\tfrac{1}{4}n^2 - O(n)$. They have two ideas. The first is based on
Tutte's famous result that every $4$-connected planar graph admits a Hamilton
cycle~\cite{Tutte}. 
By some local transformations, they manage to transform the plane dual 
graph $G^*(\mathcal A)$ to a plane $4$-connected graph $G'$ such 
that the Hamilton path of $G'$ can be translated 
to one of $G^*(\mathcal A)$. 
Their second idea is to use the levels of the line arrangement
in order to construct a long path in a straight-forward manner. 
This is an idea which we also pursue and therefore come back to 
in more detail.
For the upper bound, they argue that a path of length 
$\tfrac{1}{3}n^2 + O(n)$ is the best possible lower bound 
for all arrangements with $n$ lines. Note that this is roughly 
$2/3$ of the faces. We will repeat and refine their argument 
in Section~\ref{sec:mono}.

In a bicolored line arrangement, a face is called \emph{bicolored} 
if the lines bounding it are not all of the same color. They show 
that the graph induced by the bicolored faces is 
(alternating-)connected, i.e. there exists an alternating dual path between any two bicolored faces. This result implies the existence of an alternating path of length $\Omega(n)$ in any bicolored arrangement with $n$ lines. This is due to the fact that each bicolored arrangement has two unbounded bicolored faces of distance at least $n$. Thus, any alternating path connecting them has length $\Omega(n)$. 
They also give arrangements with a linear upper bound: Consider a bicolored line arrangement  with $n-1$ red lines and $1$ blue line. Then, only $n$ blue edges exist and any alternating path has length in $O(n)$. 


\subsection{Results and outline}
Our results hold not only for \emph{line arrangements}, but also in the more general setting of \emph{pseudoline arrangements}. Definitions and further properties can be found in the preliminaries.

Our main result is a (up to lower order terms) tight lower bound on the length of a longest path. Section~\ref{sec:mono} is devoted to its proof.

\begin{restatable}[Uncolored Arrangement]{theorem}{mono}\label{thm:mono}
 In a simple arrangement of $n$ pseudolines in the Euclidean plane, there exist a path of length $\frac{1}{3}n^2-O(n)$.
 Indeed, there are exponentially many paths of lengt~$\Omega(n^2)$.
\end{restatable}
 
 The idea of the proof consists of three steps.
 In the first step, we define some tunnels
 using the level structure of the arrangement.
 This already gives a path of length $\frac{1}{4}n^2-O(n)$.
 In the second step, we reroute this path to incorporate
 unused faces in a sophisticated way.
 In the last step, we will show, using some 
 carefully designed charging scheme,
 that we indeed used roughly $2/3$ of all the faces and
 this finishes the proof. The second and the third step 
 are done together as our charging scheme also helps us to 
 define the rerouting.
 The proof is fairly technical, because we have to guarantee 
 that the reroutings are indeed possible, without reusing a face.

\vspace{0.1cm}

For the bicolored case, we were able to improve the upper bound, to a more balanced scenario.
\begin{restatable}[Bicolored Arrangement]{theorem}{upbi}\label{thm:upbi}
 There exists a simple arrangement of $3k$ red and $2k$ blue lines where any alternating dual path has length of at most $14k$, for every odd $k$.
\end{restatable}
The idea of this construction is to distinguish between 
\emph{ monochromatic faces} and \emph{bicolored} faces, 
depending on whether
the lines bounding it have all the same color or not.
Note that an alternating path must not traverse through
a monochromatic face.
We use monochromatic faces in order to separate the bicolored
ones in a way that the graph induced by the bicolored faces 
is similar to a star with a cycle instead of a center vertex.
Section~\ref{sec:UBE} explains the example in detail.

\vspace{0.1cm}

Section~\ref{sec:ran} uses some ideas of the uncolored case 
to show that a random coloring of any pseudoline arrangement 
has an alternating dual-path of length $\Omega(n^2/\log n)$ 
with high probability. This implies that there is a 
coloring for any pseudoline arrangement such that it allows 
for an alternating path of this length.

\begin{restatable}[Random Bicolorings]{theorem}{RC}\label{thm:probLongPath}
    In a random bicoloring of an arrangement of $n$ pseudolines, there
    exists an alternating path of length
    $\Omega \left( \frac{n^2}{\log n} \right)$ with high probability.
\end{restatable}

The idea of the last theorem is to use tunnels with logarithmic
 width. This is enough to find a path from left to right in 
 each tunnel with high probability.

\subsection{Motivation and more related work}
Besides the simplicity of the problem formulation, the main motivating reason to study long paths in uncolored and bicolored line arrangements is its close connection to 
points in the plane.
In computational and discrete geometry  (and other fields) line arrangements are interesting objects of study and their properties have therefore been studied for decades~\cite{felsner}. 

One fact which makes line arrangements interesting is that they are in bijection to fundamental objects: points in the plane.  The bijection exchanges coordinates of a point with the slope and intercept of a line: $(a,b)\longleftrightarrow y=ax-b$. This bijection preserves incidences and order. It is called dual transformation. 
Due to this duality, knowledge from one setting can easily be transformed to the other and vice versa. However, certain properties are more easily captured or expressible by one of the settings. Thus, systematically studying the graph properties of line arrangements may lead to important insight to basic objects of the plane.
Whereas configurations with red and blue points are already well-studied, studying 
bicolored line arrangements is a more recent trend.

Finding long paths or cycles in the primal graph of arrangements has been explicitly studied by Felsner, Hurtado, Noy and Streinu~\cite{HamStefan}.
They have a series of simple results and they already had the idea of routing
paths along the levels of their arrangements.

We call a face in a bicolored linearrangement \emph{bicolored}, if it is bounded
by lines of both colors and \emph{monochromatic} otherwise. It is clear that any alternating path can visit at most two monochromatic faces. 
Thus bicolored line arrangements minimizing the number of bichromatic faces
are potential upper bound examples for the longest path~\cite{DBLP:journals/dmtcs/BoseCCHKLT13}.

A variant, which apparently inspired the study of our problem, was the 
study of a longest non-crossing alternating paths on a given set of red 
and blue points. For an overview of results in this direction see~\cite{ViolaPHD}.

A preliminary version of this work appeared in~\cite{LongPathEuroCG}.

\subsection{Preliminaries}\label{sec:def}

\begin{figure}[bhtp]%
\centering
\includegraphics{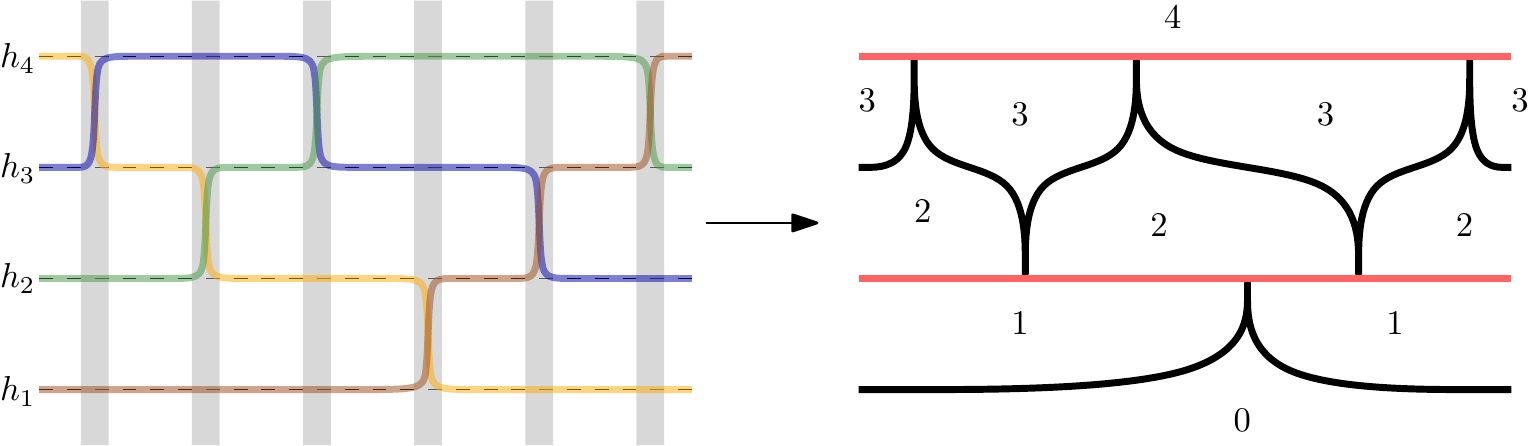}%
\caption{Left: A wiring diagram is displayed. The pseudolines are colored 
orange, blue, green and brown. The horizontal lines 
$h_1,\ldots, h_4$  are dashed and the vertical stripes 
are gray. Right: An equivalent tunnel diagram with width $w = 2$ is displayed. Wall edges are drawn in light red and tunnel 
edges in black. Every face is labeled by its level.}%
\label{fig:WiringDef}%
\end{figure}

An \emph{(Euclidean) pseudoline arrangement} is a set of $x$-monotone curves such that each pair of pseudolines crosses exactly once.
A pseudoline arrangement is \emph{simple} if no three pseudolines have a common point of intersection.
Any simple line arrangement is also a simple pseudoline arrangement.
A pseudoline arrangement partitions the plane into cells of dimension 0, 1, and 2 -- the \emph{vertices}, \emph{edges} and \emph{faces}.
We consider sequences of faces of the arrangement such that
consecutive faces share an edge and no face appears more than once.
We refer to such sequences as \emph{dual paths} and define the
\emph{length} of a dual path to be the number of involved faces. 
We drop the dual, when its clear from the context, but emphasize
primal and dual, when it might cause confusion. Dual paths can 
be seen as simple path in the graph $G^*(A)$.

We now summarize some useful facts about arrangements and 
their combinatorial representations. 
Clearly, the faces of a pseudoline arrangement $\mathcal A$ can be partitioned into \emph{bounded} and \emph{unbounded} faces.
The faces of $\mathcal A$ can be bicolored, such that any 
two faces that are incident to a common edge have different 
colors, i.e., the dual graph is bipartite.
A simple pseudoline arrangement in the Euclidean plane 
has $\binom{n}{2}$ vertices, $n^2$ edges and 
$\binom{n}{2} + n + 1$ faces. The number of unbounded 
faces is $2n$.
For later reference, we summarize two more facts in Observation~\ref{obs:Arr2},
which follow immediately from the fact that two lines cross at most once. 
\begin{observation}\label{obs:Arr2} Properties of Simple Arrangements
 \begin{enumerate}[label=\alph*),itemsep=0pt]
  \item \label{obsItem:deg3} Bounded faces are at least of degree 3.
  \item \label{obsItem:TriagNoAdj} No two bounded faces of degree 3 are adjacent. 
 \end{enumerate}
\end{observation}


We say two arrangements are \emph{equivalent}, if there is 
an isomorphism which preserves the faces, edges and vertices.
It is standard to represent pseudoline arrangements by an 
equivalent \emph{wiring diagram} (introduced by
Goodman~\cite{Goodman}). 
Let $h_1,\ldots, h_n$ be the horizontal lines with $h_i(x)=i$.
An arrangement of $n$ pseudolines is a \emph{wiring diagram} $W$ 
if 
\begin{enumerate}[itemsep=0pt]
	\item the pseudolines are almost always on the horizontal lines 
$h_1,\ldots,h_n$ and 
\item there are $\binom{n}{2}$ thin vertical disjoint stripes such that 
\item in each stripe two pseudolines cross and
\item this is the only time a pseudoline is not on a horizontal line.
\end{enumerate}
See also Figure~\ref{fig:WiringDef} for an illustration.
For an introduction to wiring diagrams, we refer 
to \cite{felsner} and \cite{goodmanHandbook}.\\

It is also possible to draw more than one crossing in a 
horizontal stripe if these crossings are independent, see Figure~\ref{fig:UnboundedFaces}.

Orienting the lines from left to right, induces an acyclic orientation of the edges. We call the unbounded face above $h_n$ \emph{top face} and the one below $h_1$ \emph{bottom face} of $W$. All the other unbounded faces are either \emph{left} or \emph{right} unbounded. For every (oriented) pseudolines $\ell$, the top face of $W$ is to the left of $\ell$ and the bottom face to the right of $\ell$.

The \emph{$i$-th level} of a pseudoline arrangement $\A$ is the set
 of edges with exactly $i$ lines strictly below. 
The edges of level $i$ form an $x$-monotone curve in 
every pseudoline arrangement.

We define the \emph{(face)-level} of a 
 face $f$ a the number of lines strictly below $f$.
 It is easy to check that this is well defined.
 The bottom face $f_{\text{bottom}}$ of a wiring diagram $W$ is 
 the unique face of level $0$.
Alternatively, one could define the level of a face $f$ as the length
of the shortest dual path from $f$ to
$f_{\text{bottom}}$ (Here and only here, 
the length of a path is the 
number of involved edges.).

 Note that for each unbounded face there exists a 
 shortest path to $f_\text{bottom}$
that consists only of unbounded faces. 
Hence, there are exactly two unbounded faces in each level; 
in $L_i$, we denote the left one by $l_i$, 
and the right one by $r_i$, see Figure~\ref{fig:UnboundedFaces}.
 
 \begin{figure}[hbtp]
 \centering
 \includegraphics{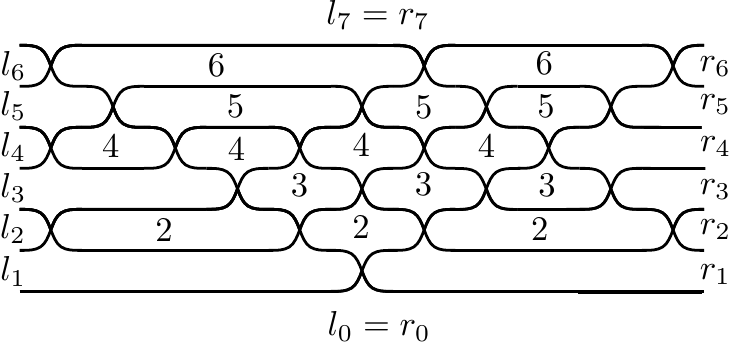}%
\caption{A Wiring diagram is displayed. Each bounded face is labeled with its level and each unbounded face is labeled with $l_i$ or $r_i$, depending on its level and its side. Some crossings are drawn
within the same stripe.}%
\label{fig:UnboundedFaces}%
\end{figure}

We denote by $L_i$ the set of faces with level $i$.
Let $w$ be some integer that divides $n+1$.
We define the $i$-th tunnel of width $w$ as
 \[\mathcal{T}^w_i=\bigcup_{j=iw}^{(i+1)w-1} \, L_j \quad 
 \text{ , for } i = 0,\ldots, \tfrac{n+1}{w}-1.\]
We denote by $l =  \tfrac{n+1}{w} \, -1 $ the index of 
the last tunnel.
To handle the case that $w$ does not divide $n+1$, note that only 
the last tunnel will be smaller and nothing else changes.
Edges bounding faces of different tunnels are called \emph{wall
 edges}. All other edges are called \emph{tunnel edges}.
It will be convenient to represent a pseudoline arrangement by a 
\emph{tunnel diagram}. In a \emph{tunnel diagram} each tunnel
 $\mathcal{T}_i$ is drawn in a horizontal stripe, except for the
 first and last tunnel, which are drawn in an  half plane.
 This implies that all the wall edges separating the same tunnel
 are on one line.
 See on the right of Figure~\ref{fig:WiringDef} for an example.
 Note that the tunnel diagram of a wiring diagram has the same 
 level structure as long as the bottom face is the same.

\section{Long Paths in Pseudoline Arrangements}\label{sec:mono}
In this section we prove Theorem~\ref{thm:mono}:
\mono*

Before we start, we note that there  are line arrangement with $n$ lines, introduced by F\"uredi and Pal\'asti \cite{furedi},
in which the longest path is of length $\frac{1}{3}n^2+O(n)$. 
Hence, up to lower order terms Theorem~\ref{thm:mono} is tight.
For more details, see the remark at the end of this section.\\

We start with a brief overview of the proof.
In the first step of the path construction, we find a path in each tunnel and connect all paths in a consistent manner at their ends. This gives a path~$P$ of length $\frac{1}{4}n^2-O(n)$.
To strengthen this result we prolong the current path~$P$ by incorporating sufficiently many faces, which are not  yet used by~$P$.
We call these unused faces \emph{bad}. (For a precise definition, see bellow.) The set of bad faces has strong structural properties. Most importantly, the set of bad faces induces a set of paths. 
These paths can be incorporated until only isolated bad faces remain.
This has to be done carefully, due to two reasons:
Firstly, some unbounded faces cannot be incorporated. Secondly, after one
rerouting, the structure of our path changes and a second rerouting may not be
possible. After eliminating adjacent bad faces, every remaining unused bad face is
given two units of charge. The charged faces distribute the charge to traversed
faces. It is possible that some traversed faces obtain two units of charge. In
these cases, we reroute again or redistribute the charge. For this second round of
rerouting, we identified the specific configurations that can appear and prove that
no other situation occurs. At last, we will conclude using the charging scheme that
roughly two third of all the faces are traversed and this will conclude the proof.


In this section, we work with tunnels of width 2, i.e. the {\em tunnel} $\mathcal T_i$ is the union of the two levels $L_{2i}$  and $L_{2i+1}$ 
for $0\leq i\leq \lfloor\frac{n-1}{2}\rfloor$.
The set of edges separating tunnels $\mathcal T_{i-1}$ and $\mathcal T_{i}$ is called {\em top wall} of tunnel $\mathcal T_i$; the one separating $\mathcal T_i$ and $\mathcal T_{i+1}$ {\em bottom wall} of tunnel $\mathcal T_i$. 
An edge $e$ of a face $f$ is a {\em wall edge} if $e$ is part of a wall, otherwise it is a \emph{tunnel edge}, see also Figure~\ref{fig:tunnel}. Depending on the type of shared edge, two adjacent faces are \emph{tunnel neighbors} or \emph{wall neighbors}.
\begin{figure}[htb]
 \centering
 \includegraphics{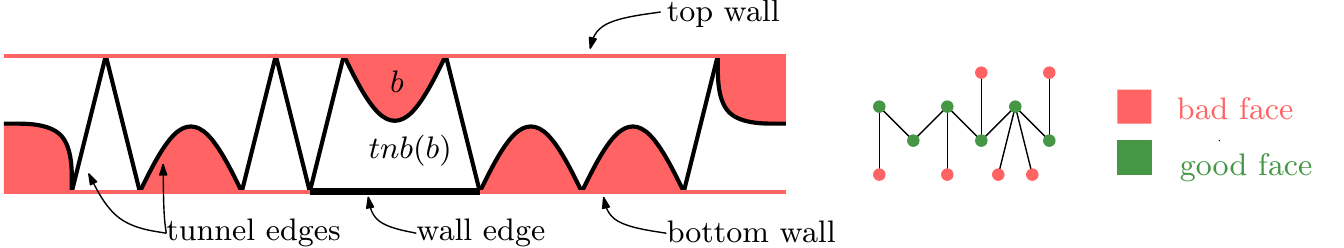}
 \caption{Schematic tunnel structure in primal (left) and dual view (right).}
 \label{fig:tunnel}
\end{figure}

Observe that all vertices (which are of degree 4) lie on some wall and are adjacent to exactly two wall edges and two tunnel edges which do not alternate. 
 A face is called \emph{bad} if all its vertices lie on the same wall, otherwise it is called \emph{good}.
 A path is called \emph{bad path} if all involved faces are  bad. 
 Further, every face has either wall edges to the top or to the bottom, but never both. 
For every wall, its edges can be oriented from the left unbounded to the right unbounded edge. This orientation induces a linear order of the edges of a wall (and on the vertices). We call this the \emph{wall order}.
For two faces $f$ and $f'$ adjacent to the same wall, $f$ is left of $f'$ if its leftmost wall edge is left of the one of $f'$. (Since each face has wall edges on exactly one wall, this is well defined.)

The set of tunnel edges belonging to a tunnel can be interpreted as a curve, also \emph{tunnel curve}, that runs from left to right. If the curve touches the same wall twice in a row, the face, which is bounded by the wall and the curve, is bad. Hence, a bad face $b$ has exactly one tunnel edge, yielding  a unique \emph{tunnel neighbor} $tnb(b)$. If the curve connects the top and bottom wall, the adjacent faces of this edge are both good.
The dual graph of a tunnel is a caterpillar where the backbone is formed by good faces which alternatingly have top and bottom wall edges. The leaves are bad faces. 
These tunnel properties are summarized in Observation~\ref{obs:strucProp}.

\begin{observation}\label{obs:strucProp}
 Structural tunnel properties
 \begin{enumerate}[label=\alph*),itemsep=0pt]
  \item Every face $f$ has a wall edge; all wall edges of $f$ are either top or bottom wall edges of $f$'s tunnel.
  \item If a good and bad face share a tunnel edge, then their wall edges belong to different walls.
  \item Every bad face $b$ has exactly one tunnel neighbor $tnb(b)$ and $tnb(b)$ is good.
  \item Bad faces within the same tunnel are not adjacent.
  \item There are two unbounded faces at the beginning and at the end of the tunnel, one of which is good and the other is bad.
  \item There exist a unique path between each two faces of a tunnel.
  \item \label{itm:backbone} The graph $G[\mathcal T_i]$ is a caterpillar. Its backbone is formed by good faces which alternatingly have top and bottom wall edges. The leaves are bad faces. 
 \end{enumerate}
\end{observation}

Besides the tunnel properties, we need a property which is based on simplicity of the arrangement. 
Given two faces $f$ and $f'$, we say $f$ is {\em nested} inside $f'$ if all wall edges of $f$ are also wall edges of $f'$. 

\begin{observation}\label{obs:nested} 
Bounded bad faces are not nested within another face.
\end{observation}
This observation follows from the fact that a nested bad face has only two adjacent faces, which is a contradiction to the fact that bounded faces are at least of degree 3, see Observation~\ref{obs:Arr2}\ref{obsItem:deg3}.\\

\bigskip

We now define an ordered set of tunnel paths $\mathcal P:=\{P_i\mid i\in[\lfloor \frac{n}{2} \rfloor]\}$ where
$$P_i:=\text {path in tunnel } \mathcal T_i\ \begin{cases}
       \text{from } l_{2i-1} \text{ to } r_{2i}, & i \text{ odd,}\\
       \text{from } l_{2i} \text{ to } r_{2i-1}, & i \text{ even.}
      \end{cases}$$
      
\begin{figure}[ht]
\centering
 \includegraphics[width=.5\textwidth]{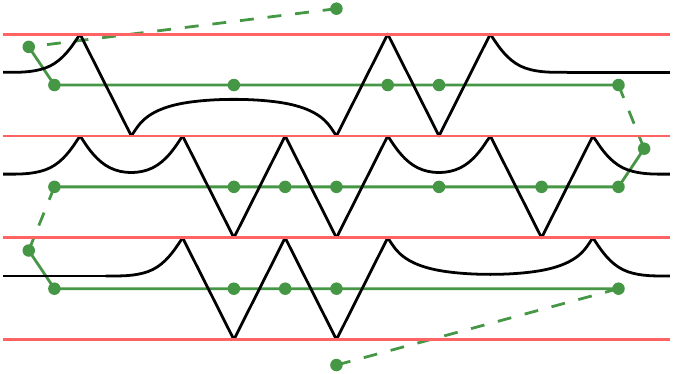}
  \caption{Tunnel diagram together with the gluable paths family $\mathcal P$.}
  \label{fig:wiring}
\end{figure}
By Lemma~\ref{obs:strucProp}~\ref{itm:backbone} the paths $P_i$ are essentially the backbones of $G[{\cal T}_i]$ connecting some left unbounded face with some right unbounded face.      
We think of path $P_i$ as oriented from the left to the right unbounded face.
The path family $\mathcal P$ has the property of being \emph{glueable}, that is the paths are pairwisely disjoint and for odd $i$, the end faces of $P_i$ and $P_{i+1}$ are adjacent, and for even $i$, the start faces of $P_i$ and $P_{i+1}$ are adjacent. Hence,  the paths in $\mathcal P$ can be combined (glued) to one  path by attaching them alternating in a forward and reverse orientation: $P:=P_1,\overleftarrow{P_2}, P_3,\overleftarrow{P_4},\dots$
 (Here $\overleftarrow{P}$ denotes the path that traverses the vertices of $P$ in reverse order.)

\begin{claim}\label{clm:length}
 Let $Q$ be a path from $l_i$ to $r_j$ in $\mathcal A$.\\
If $i,j\leq {n}/{2}$, then $|Q|\geq i+j+1$.
If $i,j\geq {n}/{2}$, then $|Q|\geq2n-i-j+1$.
\end{claim}
\begin{proof}
The path $Q$ must cross at least the $i$ lines starting above $l_i$ and the $j$ lines ending above $r_j$.
Crossing any of these lines yields one new face.
Hence, the path is at least of length $i+j+1$. 
For $i,j\geq {n}/{2}$, consider the lines starting and ending below $l_i$ and $r_j$.
\end{proof}

Summing the length of the paths in $\mathcal P$, we immediately obtain that  $P$ is of quadratic length.
\begin{claim*}
 $P$ is at least of length $n^2/4-O(n)$.
\end{claim*}
\vspace{-12pt}
$$\sum_{i=1}^{\lfloor\frac{n}{2}\rfloor}|P_i|\geq 2\sum_{i=1}^{\lfloor\frac{n}{4}\rfloor} 4i\geq
n^2/4 -n$$
Since the set of unbounded faces induces a cycle, we can combine each subset of paths to one long path. Choosing all subsets of size $\lfloor\frac{n}{4}\rfloor$, we obtain  $ \binom{\lfloor n/2\rfloor}{\lfloor n/4\rfloor}\approx\frac{2^{n/2}}{\sqrt{n}}$ many paths of length $\Omega(n^2)$.

\vspace{1cm}

In the following, we prolong the path $P$ in order to improve the leading coefficient of the length bound from $1/4$ to $1/3$.
A path (or a glueable path family) partitions the set of faces into \emph{traversed} and \emph{not traversed} faces. Since unbounded faces turn out to be special cases, we need to treat them separately. Let $F$ denote the set of all faces of $\mathcal A$ and $U\subset F$ the unbounded faces.
Given a glueable path family $\mathcal{Q}$,
we partition the bounded faces of $F$ into 
the set $T$ of bounded faces traversed by $\mathcal Q$, and
the set $N$ of bounded faces not traversed by $\mathcal Q$. 
To simplify notation: for a glueable path familily $\mathcal Q^{x}$, we denote these sets as  $T^{x}$ and $N^{x}$.
A path family $\mathcal Q'$ is a \emph{valid rerouting} of a glueable path family $\mathcal Q$ if 
$\mathcal Q'$ is glueable and $T\subseteq T'$.\\


The proof concept is as follows: We start with the path family $\mathcal P$ and charge each face in $N$ by 2 units, i.e. $ch(f)=2$ for all $f\in N$.
By a sequence of discharging and valid rerouting steps,  we will obtain a final path family $\mathcal P^*$ and final charge function $ch^*$ such that:
\begin{enumerate}[itemsep=0pt,label=({\arabic*})]
 \item $ch^*(f)=0$ for all $f\in N^*$. 
 \item $ch^*(f)\leq 1$ for all $f\in T^*$.
 \item $ch^*(f)\leq 2$ for all $f\in U$.
 \item $\sum_f ch^*(f)=2|N^*|$.
\end{enumerate}

These conditions imply the wished result as follows:
 By definition,  
 $|F|=|T^*|+|N^*|+|U|$, with  $|F|=\frac{n(n+1)}{2}+1$ and $|U|= 2n$. 
  The conditions give $2|N^*|\leq |T^*|+2|U|$. Clearly, the final path $P^*$ contains all bounded traversed faces:
 $$|P^*|\geq |T^*|\geq  2|N^*|-2|U|= 2(|F|-|T^*|-2|U|) \implies |T^*|\geq \frac{2|F|}{3}-\frac{4|U|}{3}=\frac{n^2}{3}-O(n)$$
 In other words, these conditions imply that $2/3$ of the bounded faces are traversed. Since there are only $O(n)$ unbounded faces, the wished result follows.

\subsubsection*{Initial charge}
We give some initial charge  $ch:F\to \mathbb N$ to the faces in the following way: 
$$ch(f)=\begin{cases}
         2& f\in N,\\
         0&\text{else}.
        \end{cases}
$$
Hence condition $(4)$ is fulfilled by definition. This property will be maintained in the entire process, because we delete charge of faces that become traversed.
Note that charged faces are, by definition, bad and bounded.

\subsubsection*{\headrule}
A charged face (bad and bounded) sends its charge to two different faces, one unit through its leftmost wall edge and the other through the right vertex of its leftmost wall edge, see Figure~\ref{fig:discharging}.
\begin{figure}[ht]
\centering
 \includegraphics{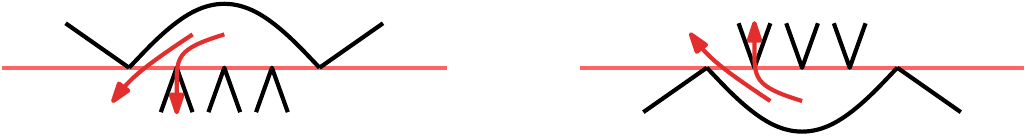}
  \caption{Discharging of bounded not traversed faces through leftmost wall edge and second leftmost vertex.}
  \label{fig:discharging}
\end{figure}
\begin{observation}\label{obs:charge}
 The charge of a face $f$ is sent to faces of an adjacent tunnel. Moreover, charge through vertices is sent to good faces.
\end{observation}
\begin{proof}
  By definition, the charge of $f$ is sent through the wall and hence, ends up in an adjacent tunnel.
  
  Let $f$ be a face sending charge to $f'$ through a vertex $v$ which is without loss of generality on the top wall of the tunnel of $f'$. By the tunnel structure, $f'$ has wall edges on the bottom wall of its tunnel. Since $v$ is on the top wall, $f'$ is a good face.
\end{proof}

Its important to note that we may divert some of this charge. This is why we think of it as ''being on its way`` that is, we distinguish between charge of a face which is \emph{sent} and \emph{obtained}. In particular, we  partially \emph{overwrite} this \headrule\ in the following two phases!

\subsubsection*{Rerouting and Discharging Step 1}
 First, we guarantee that no face in $N$ obtains charge.
 In particular, we eliminate adjacent charged faces by incorporating them in some path. If this is not possible, we discharge. 
 Consider the  set of (all not only bounded) not traversed faces $N_\text{All}$ in $\mathcal P$.

We exploit some structural properties of the graph $G[N_\text{All}]$ induced by the set $N_\text{All}$.
 \begin{observation}\label{obs:strucProp2}
  Structural properties of $G[N_\text{All}]$
  \begin{enumerate}[label=\alph*),itemsep=0pt]
   \item If face $f$ is good, then $f\notin N_\text{All}$.
   \item If face $f$ is bad and $f\notin N_\text{All}$, then there exists $i$ such that $P_i$ starts or ends in $f$.
   \item The graph $G[N_\text{All}]$ is a collection of paths. The wall orders induce a left-right orientation on the paths. 
   \item Let $Q=b_1,b_2,\dots$ be a maximal path in $G[N_\text{All}]$ oriented from left to right. If deg$_\mathcal 
   A (b_1)\geq 3$, then the penultimate wall neighbor of $b_1$ is traversed.
   \item Let $Q=b_1,b_2,\dots$ be a maximal path in $G[N_\text{All}]$ oriented from left to right. If $|Q|\geq 5$, then there exist $j\leq 4$ such that $b_j$ has a traversed wall neighbor.
  \end{enumerate}
 \end{observation}

 \begin{proof}
 a) Let $f$ be a good face, then $f$ is part of the backbone of some caterpillar and hence, included in some $P_i$.\\
 b) Let $f$ be a bad face. Then $f$ is a leaf of some caterpillar. A leaf is included if and only if some $P_i$ starts or ends in $f$.\\
 c)  By Observation~\ref{obs:strucProp}d), all edges in $G[N_\text{All}]$ are dual to wall edges. By Observation~\ref{obs:nested}, the maximal degree in $G[N_\text{All}]$ is 2.  
 By Observation~\ref{obs:strucProp}a), all wall edges of a face are on the same wall.
 Hence, all faces of a path in $G[N_\text{All}]$  have their wall edges on the same wall, which is also cycle-free. 
 Consequently, $G[N_\text{All}]$ is a collection of paths.  
 Moreover the wallorder induces an orientation on the path: Every two faces of the path have wall edges on the same wall, hence, one is left of the other.\\
 d) Since $b_1$ is a bad face, it has only one tunnel neighbor. Since deg$_\mathcal A (b_1)\geq 3$, $b_1$ has at least two wall neighbors. By the maximality of $Q$, all but $b_2$ are traversed. Otherwise $Q$ could be extended.\\
 e) Suppose by contradiction that each of $b_1,\dots,b_4$ has no traversed wall neighbor. If deg$_\mathcal A (b_1)\geq 3$, then $b_1$ has a traversed  wall neighbor by d). Hence, deg$_\mathcal A (b_1)=2$, in particular $b_1$ is unbounded. By construction of the path family $\mathcal P$, a path in $G[N_\text{All}]$ cannot start and end with unbounded faces. Therefore $b_3$ and $b_4$ are bounded, wall neighbors and have no good wall neighbor (by assumption). This implies that they are adjacent bounded faces of degree 3; a contradiction to Observation~\ref{obs:Arr2}\ref{obsItem:TriagNoAdj}.
 \end{proof}

\begin{step}
 We construct a valid rerouting $\mathcal P^{(1)}$ of $\mathcal P$ and discharge such that 
 \begin{itemize}[itemsep=0pt]
  \item[(P1)] $ch(f)=0$ for all $f\in N^{(1)}$,
  \item[(P2)] $ch(f)\leq 2$ for all $f\in F$, and
  \item[(P3)] $\sum_f ch(f)=2|N^{(1)}|$.
 \end{itemize}
 \end{step}

Let $Q$ be a maximal path in $G[N_\text{All}]$ with $|Q|\geq 2$, where $b_1,b_2$ denote the first two (bad) faces of $Q$. The aim is to replace single edges of a current path in $\mathcal P$ by a path through not yet traversed faces.
For this replacement, we  enter the bad path by a traversed wall neighbor of a bad face. We use the sufficient conditions for the existence of such neighbor by Observation~\ref{obs:strucProp2}d) and e) and distinguish three cases, see Figure~\ref{fig:Step1}.
We initialize $\mathcal P_\text{curr}:=\mathcal P$ and redefine some of the paths during the phases of step~1.

\begin{step*}[\bfseries 1a]
  If deg$_\mathcal A (b_1)\geq 3$ and $|Q|\geq 2$:\\
  We set $f_\text{enter}$ to the penultimate wall neighbor of $b_1$ and  $f_\text{exit}$ to the tunnel neighbor of $b_2$. Let $P_i^\text{curr}$ be the path of $\mathcal P_\text{curr}$ containing $f_\text{enter}$. (We will show that $f_\text{enter}$ and $f_\text{exit}$ are consecutive in $P_i^\text{curr}$.) 
  The idea is to insert $b_1$ and $b_2$ in between $f_\text{enter}$ and $f_\text{exit}$ in $P_i^\text{curr}$. For the formal definition,  
  let $P'$ denote the prefix of $P_i^\text{curr}$ ending right before $f_\text{enter}$ and $P''$ denote the suffix of $P_i^\text{curr}$ starting after $f_\text{exit}$. (If $f_\text{enter}$ and $f_\text{exit}$ are adjacent, it holds that $P_i^\text{curr} = P', f_\text{enter},f_\text{exit},P''$.) We alter $P_i^\text{curr}$ by inserting $b_1,b_2$:
  $$ P_i^\text{curr}:=P',f_\text{enter},b_1,b_2, f_\text{exit},P''$$
  Since $b_1,b_2$ are now traversed, their status switches from $N$ to $T$ and (if they are charged,) their charge is deleted.\\
  If $(|Q|-2)\geq 2$, apply Step\,1a) to $Q':= Q-\{b_1,b_2\}$.
\end{step*}
\begin{step*}[\bfseries 1b]
  If deg$(b_1)=2$ and $|Q|\geq 5$: \\
  Determine $j_{\min}$, the smallest $j$ such that $b_j$ has a traversed left wall neighbor.
  $f_\text{enter}$ is set to the rightmost traversed wall neighbor of $b_{j_{\min}}$ and $f_\text{exit}$ as tunnel neighbor of $b_{(j_{\min}+1)}$. Let $P_i^\text{curr}$ denote the path of $\mathcal P_\text{curr}$ containing $f_\text{enter}$.
  Let $P'$ denote the prefix of $P_i^\text{curr}$ ending right before $f_\text{enter}$ and $P''$ denote the suffix of $P_i^\text{curr}$ starting after $f_\text{exit}$.
  We reroute as in a):
  $$ P_i^\text{curr}:=P',f_\text{enter},b_{j_{\min}},b_{(j_{\min}+1)},f_\text{exit},P''$$
  Since $b_{(j_{\min})},b_{(j_{\min}+1)}$ are now traversed,  their status switches from $N$ to $T$ and (if they are charged) their charge is deleted.\\
  If $j_{\min}=4$, the total charge of $b_3$ is terminatory sent to the unbounded face $b_1$.\\
  If $(|Q|-j_\text{min} -1)\geq 2$, apply Step\,1a) to $Q':= Q-\{b_1,\dots,b_{(j_{\min}+1)}\}$.
\end{step*}
\begin{step*}[\bfseries  1c]
  If deg$(b_1)=2$ and $|Q|\leq 4$:\\
  The total charge of $b_3$ and $b_4$ (if they exist and are charged) is terminatory sent to the unbounded faces $b_1$ and $b_2$, respectively.
\end{step*}

We denote the path family obtained after handling all path in $G[N_\text{All}]$ by $\mathcal P^{(1)}$, and the corresponding set of traversed and not traversed bounded faces by $T^{(1)}$ and $N^{(1)}$, respectively.

  \begin{figure}[htb]
  \centering
    \includegraphics{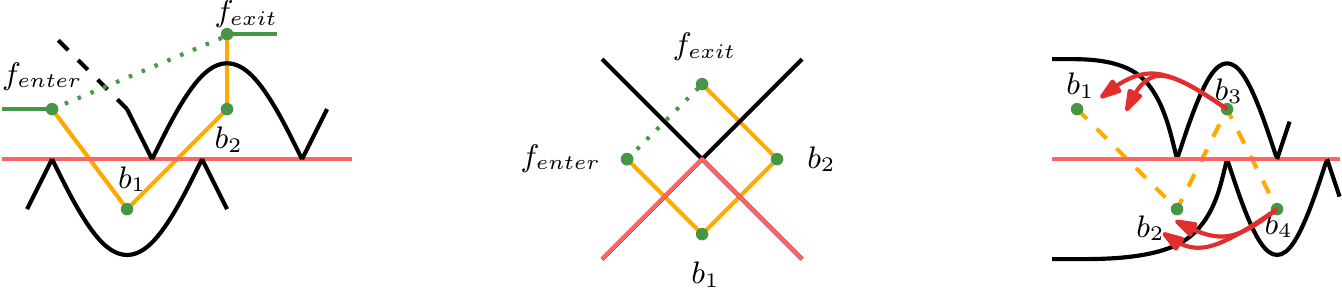}
    \caption{Left: Rerouting in Step 1a) and 1b) -- adjacent bad faces are included by replacing the dotted green edge by the orange marked path. Middle: Schematic rerouting of Step 1\,[a),\,b)]: $f_\text{enter}$, $f_\text{exit}$ $b_1, b_2$ are adjacent to a vertex. In the rerouting step an edge is replaced by the `detour' around the vertex.
    Right: Discharging in Step 1c) -- charge from $b_3$ and $b_4$ is sent to unbounded faces.}
  \label{fig:Step1}
\end{figure}
\begin{claim*}
 $\mathcal P^{(1)}$ is a valid rerouting of $\mathcal P$.
\end{claim*}
\begin{proof}
Note that by construction and Observations~\ref{obs:strucProp2}d), \ref{obs:strucProp2}e), and~\ref{obs:strucProp}c), the considered faces exist and are adjacent. Hence, the new path is well defined.

First, $\mathcal P^{(1)}$ is glueable: By construction, the first and last face of $P_i$ remained for all $i$. Moreover, the paths remain disjoint since only not yet traversed faces are incorporated.

Second, we show that  $T\subseteq T^{(1)}$. We therefore consider a single rerouting step (of Step 1a or 1b) applied to some path $P_i^\text{curr}$. For simplification, lets denote this path by $P$ before and by $P_\text{mod}$ after the modification. 
We show that  $f_\text{enter}$ and $f_\text{exit}$ are consecutive in $P$. This implies an edge was replaced by a path and consequently, all traversed faces remain traversed. 

We show this is three steps:
Firstly, $f_\text{enter}$ and $f_\text{exit}$ are adjacent faces in $\A$. Secondly, $f_\text{enter}$ and $f_\text{exit}$ are adjacent in an earlier version of $P\ ^\dagger$. Thirdly, this adjacency is removed only when $b_1$ is added to $P$.

The faces $f_\text{enter}$ and $f_\text{exit}$ are adjacent since $f_\text{enter}$, $f_\text{exit}$, $b_1$, and $b_2$ are incident to the left vertex of $b_1$ and $b_2$. Recall that every vertex is of degree 4 in the primal.

 We claim that there was a point in time when $f_\text{enter}$ and $f_\text{exit}$ were consecutive in some earlier version of $P$
 \footnote[2]{Let $\mathcal S_i$ be the sequence of paths in tunnel $\mathcal T_i$ during Step~1. $P$ and $P_\text{mod}$ occur as consecutive elements in $\mathcal S_i$. An \emph{earlier version of $P$} refers to some element of $\mathcal S_i$ before $P$.}. 
 If  $f_\text{enter}$ was initially traversed then we are done; if not, then $f_\text{enter}$ is a bad face which  was incorporated in some iteration in the role of the second bad face (it was incorporated since it is now traversed). Since the tunnel neighbor of $f_\text{enter}$ is the same as the one of $b_2$, namely $f_\text{exit}$, it follows that $f_\text{enter}$ and $f_\text{exit}$ were consecutive in some earlier version of $P$. 
 
 Now, suppose by contradiction that $f_\text{enter}$ and $f_\text{exit}$ are not consecutive in $P$ any more.
 Then, this edge was replaced by an application of Step 1 where $f_\text{enter}$ was the penultimate wall neighbor of some bad face $b$. But then, by Observations~\ref{obs:strucProp}a) and \ref{obs:nested}, $b$ must be $b_1$. This is a contradiction to the fact that $b_1$ is not traversed by $P$.
\end{proof}

Next, we state a crucial property which we apply afterwards in order to prove that the claimed conditions are fulfilled.
\begin{observation}\label{obs:rightSend}
 Bad faces still sending charge were rightmost of a bad path in $G[N_\text{All}]$.
\end{observation}
\begin{proof}
Consider a bad path $p$ in $G[N_\text{All}]$.
If rerouting in Step 1a) or 1b) is applied, the first two faces $b_1$, $b_2$ of a bad path $p$ are incorporated and the remaining bad path $p-\{b_1,b_2\}$ is handled likewise by Step 1a). In Step 1c), $b_3$ and $b_4$ remain not traversed, however their charge is terminatory dealt with. Consequently, only the rightmost face of $p$ may remain not traversed and sending charge. 
\end{proof}

  \begin{figure}[ht]
  \centering
    \includegraphics{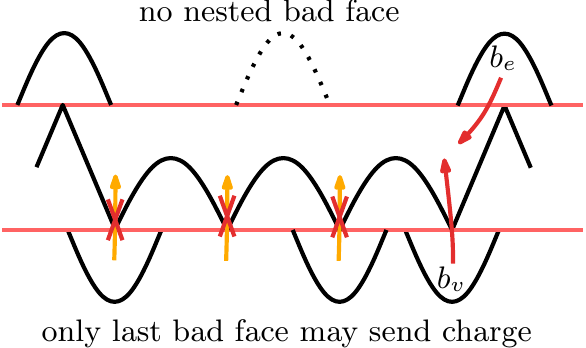}
    \caption{After Step 1, good faces obtain at most two units of charge. In particular, at most one unit  via the rightmost wall vertex and one unit via the rightmost wall edge.}
  \label{fig:ProofStep1}
\end{figure}
\begin{observation}\label{obs:2charge}
 Hence, a good face $f$ may obtain its charge through at most one vertex  and one edge.
\end{observation}
\begin{proof}
Consider Figure~\ref{fig:ProofStep1}.
   Suppose $f$ obtains charge from the two faces $b_l$ and $b_r$ through the two vertices $v_l$ and $v_r$ (on the same wall). Wlog suppose that $v_l$ is left of $v_r$ in the wall order. By the tunnel structure, the tunnel curve connects $v_l$ and $v_r$ by forming a sequence of bad faces which are tunnel neighbors of $f$. In particular, the first bad face $b$ after $v_l$ is bounded and a wall neighbor of $b_l$. That is $b\in N_\text{All}$. However, by  Observation~\ref{obs:rightSend}, $b_l$ is the rightmost vertex of a bad path in $G[N_\text{All}]$. This is a contradiction.
  
  Suppose $f$ obtains charge from two faces $b_l$ and $b_r$ through edges. Then, $f$ has two wall edges which are the leftmost wall edges of $b_l$ and $b_r$. Wlog suppose that $b_l$ is left of $b_r$.
  This implies that the bounded bad face $b_l$ is nested within $f$;
  a contradiction to Observation~\ref{obs:nested}.
  
  Hence, a good face $f$ may obtain its charge through at most one vertex $v$ and one edge $e$.
\end{proof}

\begin{claim*}
 After Step 1, (P1)--(P3) hold.
\end{claim*}

\begin{proof}
  First, we show that in Step 1 b) or 1c), charge is only sent to unbounded (not traversed) faces: By construction, charge may have been sent to the first two faces $b_1$ and $b_2$ of a bad path $p$ only if deg$_\mathcal A (b_1)=2$. Since a face of degree 2 is unbounded and has exactly two unbounded neighbors, $b_1$ and $b_2$ are unbounded (and not traversed by construction).\\[-6pt]

  Now, we show (P1), that is $ch(f)=0$ for all $f\in N^{(1)}$.
  Suppose a bounded not traversed face $f$ obtains charge. By Observation~\ref{obs:strucProp2}a), $f$ is bad. This charge is assigned by the \headrule \ or by Step 1. By the observation above, the charge is assigned by the \headrule. By Observation~\ref{obs:charge}, the bad face $f$ obtains its charge through an edge. This gives two adjacent bounded not traversed faces. By construction of Step~1, all pairs of adjacent bad faces are incorporated, unless $|Q|\leq 4$. Consequently, these  can only survive in Step 1c). Hence, $f= b_3$ of some bad path $p$ with $|p|=4$ handled in Step 1c). However, charge of $b_4$ is sent to $b_2$ and not to $b_3$, a contradiction.\\[-6pt]
 
  Next, we show (P2), that is $ch(f)\leq 2$ for all $f\in F$.
  By (P1), bounded not traversed faces are free of charge. It remains to analyze the charge of unbounded not traversed faces and traversed faces.
  
  Consider an unbounded not traversed face $f$. By Observation~\ref{obs:strucProp2}b), $f$ is bad. By Observation~\ref{obs:charge}, the bad face $f$ may obtain one unit of charge through its rightmost wall edge by the \headrule. In Step 1, unbounded not traversed faces $b_1$ and $b_2$ may obtain two units of charge from $b_3$ or $b_4$, including the unit of the \headrule which would have been sent from $b_3$ to $b_2$ in any case. Consequently, $f$ obtains at most two units of charge.
  
  Now, consider a traversed face $f$. 
  If $f$ is bad, then it obtains at most one unit of charge in \headrule\ through an edge, by Observation~\ref{obs:charge}.
  If $f$ is good, then by Observation~\ref{obs:2charge}, it may obtain its charge through at most one vertex $v$ and one edge $e$. This yields at most two units. (Note that $v$ is the rightmost vertex on its wall adjacent to $f$, and that $e$ is the rightmost wall edge of $f$.) \\[-6pt]
  
  Lastly, (P3) holds due to the fact that while updating the status of a face, switching from $N$ to $T$, its charge was deleted. Moreover, this is the only occurring status change.
\end{proof}

\subsubsection*{Rerouting and Discharging Step 2}
It remains to guarantee that faces in $T$  obtain at most one unit of charge.
\setcounter{step}{1}
\begin{step}
We construct a valid rerouting $\mathcal P^{(2)}$ of $\mathcal P^{(1)}$ and discharge such that 
 \begin{itemize}[itemsep=0pt]
  \item[(Q1)] $ch(f)=0$ for all $f\in N^{(2)}$,
  \item[(Q2)] $ch(f)\leq 1$ for all $f\in T^{(2)}$, 
  \item[(Q3)] $ch(f)\leq 2$ for all $f\in U$, and
  \item[(Q4)] $\sum_f ch(f)=2|N^{(2)}|$.
\end{itemize}
 \end{step}
 In order to obtain the wished properties, it only remains to 
consider bounded traversed faces obtaining charge of more than 1 unit. 
By Observation~\ref{obs:charge} and \ref{obs:2charge}, a (bounded) traversed face $f$ with charge of more than one unit, obtains its charge through exactly one vertex from  a face $b_v$ and one edge from a face $b_e$, which are in different tunnels, see Figure~\ref{fig:ProofStep1}. In this case, we either redistribute the charge or incorporate the two not yet traversed faces $b_v$ and $b_e$.

Let $f$ be a face in tunnel $\mathcal T_i$ then without loss of generality $b_e$ is in tunnel $\mathcal T_{i-1}$ and $b_v$ is in $\mathcal T_{i+1}$.
Denote the first and second good face after $f$ within its tunnel (in the left right order of the original $P_i$) by $f_1$ and $f_2$. The predecessor, successor and second successor of $f$ in the possibly modified $P_i$, we denote by $pr$, $s_1$, and $s_2$. 
Since $b_e$ is a face of degree $\geq 3$, these faces exist. In particular, the existence of $f_2$ implies the existence of all other faces as well.

By Observation~\ref{obs:rightSend}, $b_e$ and $b_v$ were the rightmost bad faces of a bad path. 
This implies that $f_1$ is a good face with exactly one vertex on the opposite tunnel wall from its wall edges, otherwise $b_e$ would not have been last of a bad path. 
Consequently, the edge $(f,f_1)$ was not replaced in Step 1.  Hence, $f_1=s_1$.
We make a case distinction based on whether or not $(f_1,f_2)$ is an edge in some path.

\begin{step*}[\bfseries 2a]
If  $(f_1,f_2)$ is an edge in $P_i$, then $f_2=s_2$. Let $P'$ be the prefix of $P_i$ until $pr$ and $P''$ be the suffix of $P_i$ starting after $f_2$. We redefine $P_i$ as follows: 
$$P_i:=P',pr,b_v,f_1, f, b_e,f_2,P''$$
Since $b_v,b_e$ are now traversed, their status switches from $N$ to $T$ and their charge is deleted.
\end{step*}
\begin{step*}[\bfseries 2b]
If $(f_1,f_2)$ was replaced in Step 1 (when some $b_1$,$b_2$ where incorporated), we send one unit of charge from $f$ to $b_1$.
\end{step*}
\begin{figure}[hbt]
\centering
\includegraphics[width=.85\textwidth]{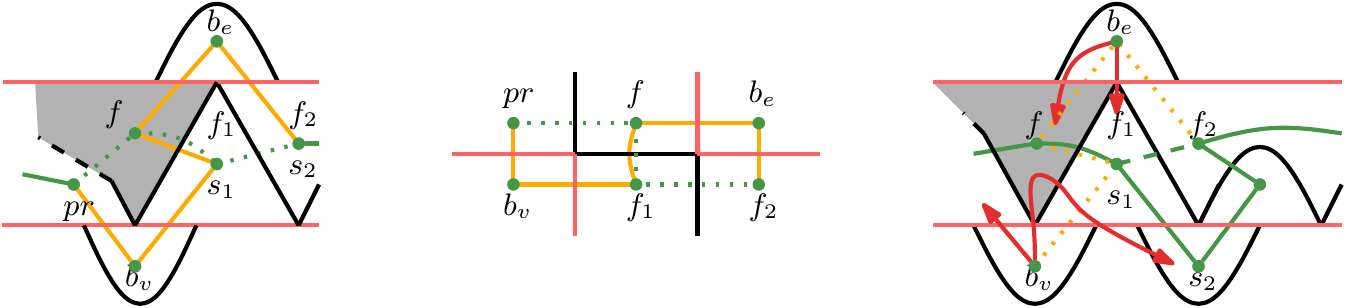}
\caption{Left: Rerouting in Step 2 -- bad faces across the tunnel are incorporated by replacing the dotted green edges by the orange marked path. Middle: Schematic rerouting of Step~2.\\
Right: Discharging in Step 2 -- One unit of $f$ (from $b_v$) is sent to $s_2$.}
\end{figure}
\begin{claim*}
 $\mathcal P^{(2)}$ is a valid rerouting of $\mathcal P^{(1)}$.
\end{claim*}
\begin{proof}
We need to argue that rerouting Step(2a) is well-defined. 
First of all, we show that all faces of the new path are adjacent. $(f_1,f)$ and $(f,b_e)$ are edges by definition. $b_v$ and $f_1$ are adjacent since $f$ and $b_v$ share no edge but a vertex $v$ (of degree 4) which is the rightmost wall vertex of $f$. Moreover,  $f_1$ is adjacent to $v$ since $f_1$ is the next good tunnel face after $f$.

The adjacency of $b_e$ and $f_2$ can be seen as follows: As discussed before, $f_1$ has exactly one vertex $v_e$ on the opposite tunnel wall from its wall edges. Hence, $v_e$ is the endpoint of the two tunnel edges of $f_1$ and hence shared with $f$ and $f_2$. Moreover, $v_e$ is the last vertex of $f$ on the wall of its wall edges. Hence, $v_e$ is also adjacent to $b_e$. As $v_e$ is of degree 4, $b_e$ and $f_2$ must be adjacent. 

It remains to show that $pr$ and $b_v$ are adjacent. 
%
%
%
We claim that $pr$ is the (penultimate) wall neighbor of $b_v$:
By Observation~\ref{obs:rightSend}, $b_v$ was the last vertex of a bad path in $G[N_{All}]$.
If $b_v$ was a bad path of length 1 in $G[N_{All}]$, then $f$ has exactly one vertex on the tunnel wall opposite from its wall edges and $pr$ is wall neighbor of $b_v$. Otherwise, the left bad neighbor of $b_v$ has been included in Step 1 as the role of $b_2$ and is now the predecessor of $f$, namely $pr$.  Thus in both cases, $pr$ and $b_v$ are adjacent.

Next, we show that neither the edge $(f,f_1)$ nor the edge $(f_1,f_2)$ was removed from $P$ earlier by Step~2.
Suppose one of the edges $(f,f_1)$ and $(f_1,f_2)$ were replaced during Step 2. By the fact that faces have wall edges on exactly one side, the interesting cases are if the edges were replaced by the opposite role from the now considered step, i.e., $(f,f_1)$ takes the role of $(f_1,f_2)$ or vice versa.
In both cases, by construction of Step 2, one of the bad faces $b_v$ and $b_e$ gets incorporated. This is a contradiction to the fact that $b_e$ and $b_v$ are not traversed.

By definition, we replace two edges by a path traversing all already traversed faces and additionally, the two not yet traversed faces. Hence, $T^{(1)}\subseteq T^{(2)}$.
Since the starting and ending face of each $P_i$ remains, and only not yet traversed faces are incorporated, $\mathcal P^{(2)}$ is glueable.
\end{proof}

\begin{claim*}
 (Q1)-(Q4) of Step 2 hold.
\end{claim*}
\begin{proof}
Note that within Step 2, charge is sent to faces which were first of a bad path in $G[N_\text{All}]$, not left unbounded, and traversed after Step 1.\\[-6pt]

Property (Q1) holds before Step 2. Within Step 2, no charge is sent to not traversed faces.\\[-6pt]

Now we show property (Q2), that is, bounded traversed faces obtain at most one unit of charge. We distinguish good and bad traversed faces.

Good faces obtaining more charge were eliminated by Step 2: Either both bad faces sending charge are incorporated and their charge is deleted or one unit of charge was redistributed.

Recall that after Step 1, bad traversed faces obtained at most one unit of charge. Consider a bad face $b$, which obtained charge in Step 2. Note that $b$ is free of charge before Step 2, since its bad neighbor is also traversed after Step 1. The sent unit  can be associated with rerouting in Step 1 and therefore with the rightmost wall edge of $b$.
Hence, it is at most one unit.\\[-6pt]

Property (Q3) held before Step 2. By the above observation, during Step~2 charge is not sent to left unbounded faces. Hence, it remains to check right unbounded bad faces. Since the extra charge of Step 1 is sent only to left unbounded faces, the above reasoning applies to right unbounded bad faces as to bounded bad faces.\\[-6pt]

We now show property (Q4). By construction, whenever a face changes its status from not traversed to traversed, its charge is deleted. Due to the valid rerouting, no other status changes occur. Hence, (Q4) holds.
\end{proof}

This establishes the claimed properties, and Theorem~\ref{thm:mono} follows as shown before.\\

\begin{remark}
	It was pointed out in \cite{cellpaths} that $n^2/3 + O(n)$ is tight. The argument goes as follows: The dual graph of a line arrangement is bipartite and there exist line arrangements \A \ (introduced by F\"uredi and Pal\'asti \cite{furedi},) with roughly $1/3$ black and $2/3$ white faces. As every second visited face must be  black the longest path has length of approximately $2/3$ of the total number of faces.
	
	However, does this upperbound also hold if the number of black and white faces is almost equal? The answer is yes. 
	To see this we define a line arrangement \A' as follows: Consider an arrangement with \A\ with roughly $1/3$ black and $2/3$ white faces and insert a black point in every black face and a white point in every white face. Perturb the points such that they are in general position, if they haven't been so already. 
	Let the line $\ell$ be a ham-sandwich cut for this set of points. Now, \A' is the arrangement \A\ with the line $\ell$ inserted. Note that the number of black faces and the number of white faces agree up to lower order terms. 
	We argue that any alternating path is still of length at most $n^2/3 + O(n)$. To see this, consider an alternating path $P$ in \A' and remove the edges that cross $\ell$ and also the end vertices to these edges. Note that we get a set of at most $n$ paths that all live in \A . Except for the last face of each dual path, half of the vertices must be black. And thus all the $n$ paths together have length at most $n^2/3 + O(n)$ and hence also the original path $P$ in \A' has length $n^2/3 + O(n)$.
\end{remark}

\section{Upper Bound Example}\label{sec:UBE}

\begin{figure}[htb]
  \centering\includegraphics[width=.85\textwidth]{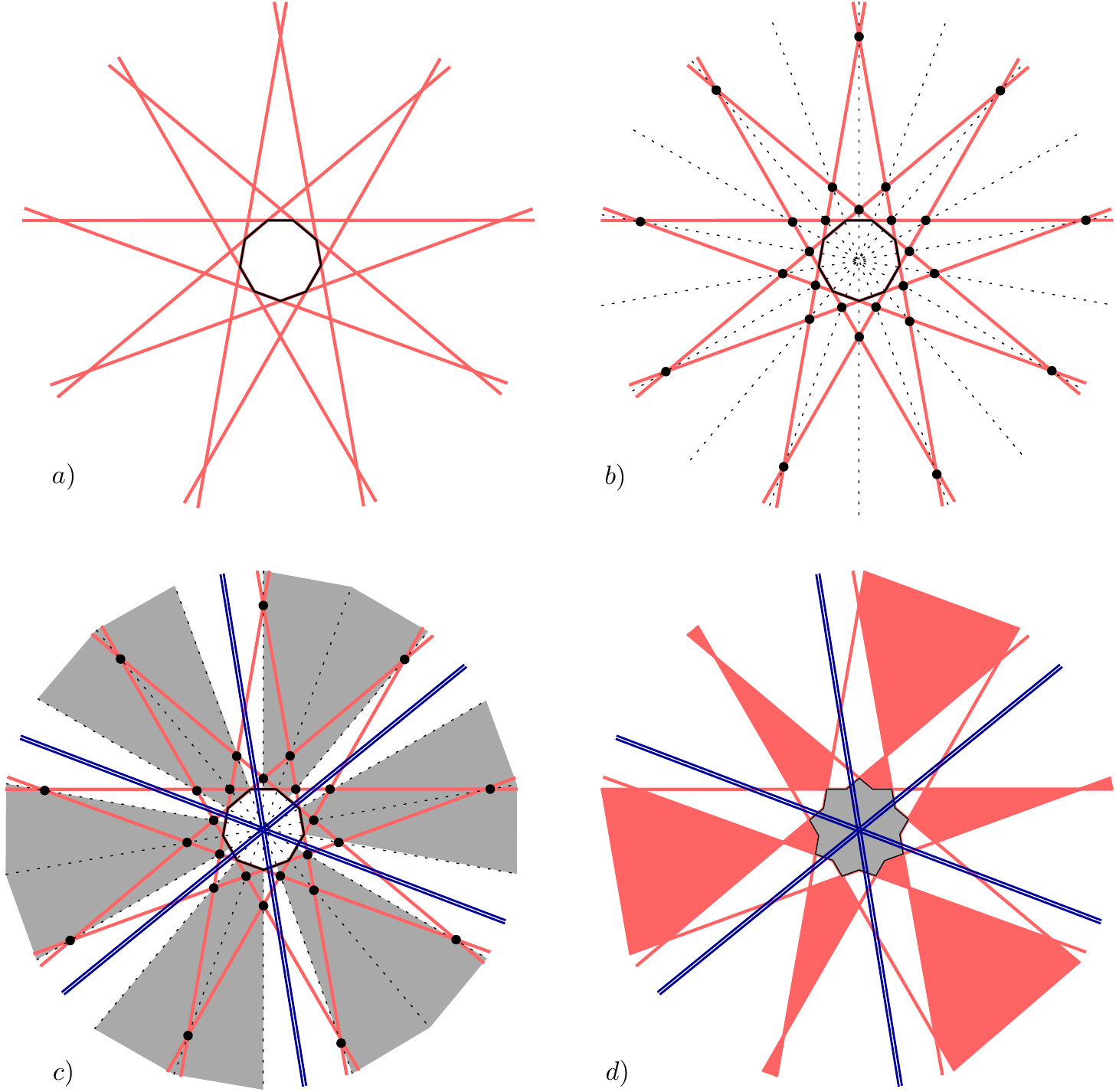}
  \caption{The red lines are the extensions of a regular $3n$-gon. The dotted auxiliary lines are added to define slabs. }\label{fig:slabs}
\end{figure}
In this section, we study upper bounds on the length of alternating paths in bicolored arrangements. In particular, 
we describe a line arrangement \A \ with $3n$ red and $2n$ blue lines with no alternating path longer than $14n$.
The main idea of the construction is to separate blue pairs of lines by red lines. 

\upbi*


\begin{proof}
We first describe the red arrangement and show some interesting properties of it.
Let $k=3n$.  See Figure~\ref{fig:slabs} for an illustration of the construction. 
The red arrangement is the extension of the sides of a regular $3n$-gon. We require $k$ and consequently $n$ to be odd.
Thereafter we draw $k$ auxiliary dotted lines, which are not part of the arrangement. Each dotted line goes from one vertex of the $k$-gon through the midpoint of the opposite side of the $k$-gon. In particular, for any two sides of the $k$-gon there exists a dotted line bisecting it. 
The red lines together with the dotted lines are, up to projective transformation, the Böröczky-example, which minimizes the number of ordinary crossings~\cite{crowe1968sylvester,tao}. 
In the following we state some observations. For the delight of the reader we also repeat observations that might be well-known.

\begin{observation}\label{clm:slabs}
For each red line holds, that whenever it crosses a dotted line it also crosses a red line. Except once when the dotted line crosses the side of the $k$-gon.\end{observation}
\begin{proof}
 The intersection point of any two red lines lies on a dotted line. This can be seen by considering the bisector of the two red lines, which is a dotted line. Every red line intersects $k-1$ other red lines.
 Every red line is crossed at $k$ different points by a dotted line. This holds because the considered lines are not parallel.
 The only point a red line crosses a dotted line without crossing a second red line simultaneously is on the regular $k$-gon.
 \end{proof}
 
 A \emph{slab} is one of the $2k$ maximal unbounded regions formed by the $n$-gon and the dotted lines, see Figure~\ref{fig:slabs}~$c)$. We mark every third of them and refer to them as \emph{marked slabs}. Here we use that $k$ is divisible by three.
 The \emph{middle part} of the construction is the union of red faces marked gray in Figure~\ref{fig:slabs}~$d)$.
 The blue lines are inserted as (almost) parallel pairs called \emph{twins} such that they cross dotted lines within the original $k$-gon and they lie entirely within two opposite marked slabs. 
 We perturb the blue lines slightly such that they are not parallel. 
Within a slab we can order all faces according to the distance of the original $k$-gon. This ordering defines naturally two directions for a dual path.
 \begin{figure}[htbp]
  \centering\includegraphics[width=.85\textwidth]{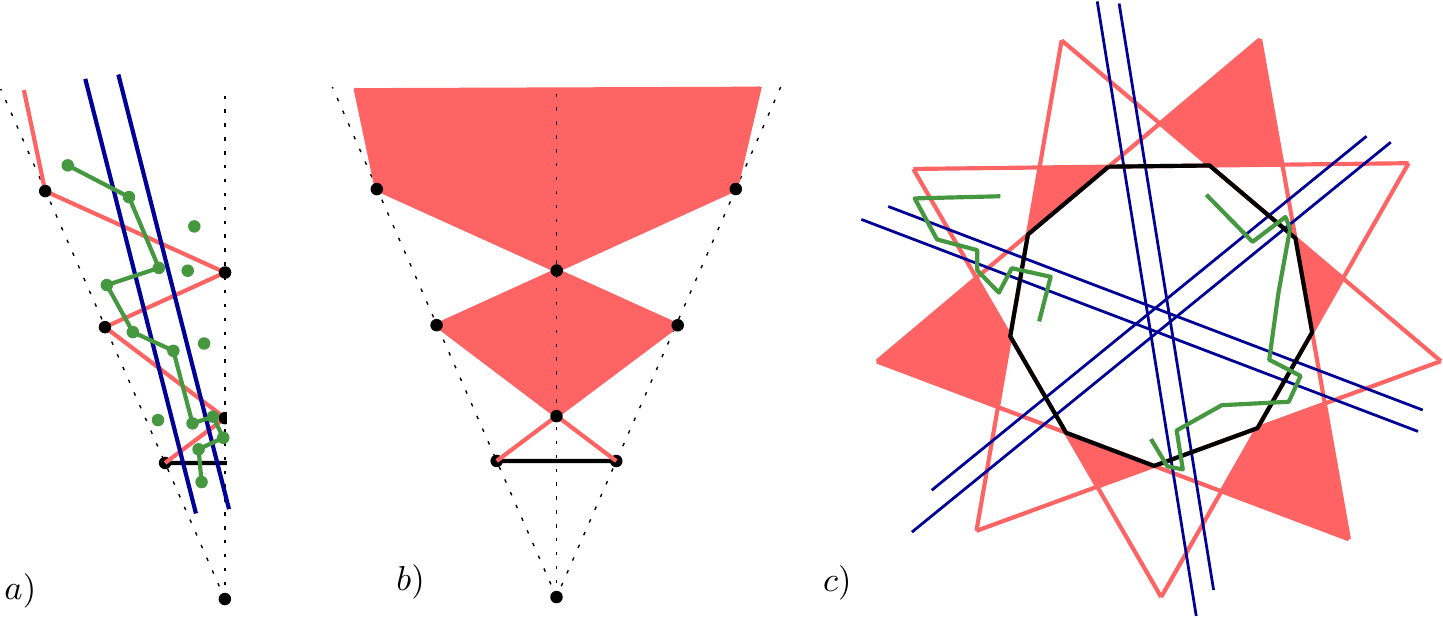}
  \caption{a) the marked slab with an alternating  dual  path. b)  separation of marked slabs by red faces. c) zoom in of the middle part with various alternating   dual paths.}\label{fig:ray}
\end{figure}

 \begin{observation}\label{clm:chdir} Any alternating  dual path within a marked slab cannot change its direction.
 \end{observation}
\begin{proof}
Consider Figure~\ref{fig:ray}~$a)$. For every blue edge exists a blue twin edge, bounding same face $f$. 
It is easy to see that any alternating  dual  path needs to use  exactly one of the twin edges and thus the face $f$. A return  dual path would need to use  this face as well.
 \end{proof}
 
 The following observation is the essence of the construction.
 \begin{observation}\label{clm:raynoReturn}
 	 Any alternating  dual path in a marked slab must only leave to the middle part and cannot  switch to another marked slab.
 \end{observation}
 \begin{proof}
  See Figure~\ref{fig:slabs}~d) and Figure~\ref{fig:ray}~b). Every two neighboring marked slabs are separated by red faces. 
Every face intersected by the middle dotted line is bounded by the left and right dotted line as no red line can cross a dotted line without crossing another red line at the very same point, by Claim~\ref{clm:slabs}. 
\end{proof}
 
\begin{observation}\label{clm:raylength}
	Any alternating  dual path in a marked slab has length at most $k = 3n$.
\end{observation}
\begin{proof}
	Every red line intersects $k-1$ slabs and we have $2k$ slabs and $k$ red lines.
	As every slab is symmetric to all the others we have $(k-1)/2$ red lines per slab.
	Any alternating  dual path in a marked slab crosses every red line at most once.
	\end{proof}

\begin{observation}\label{clm:midlength}
	Any alternating  dual path in the middle part has length at most $8n$.
\end{observation}
\begin{proof}
	Every blue line is intersected at most twice and we have $2n$ blue lines. 
	Note that the red faces marked in Figure~\ref{fig:ray}c) must be avoided.
\end{proof}
Any  dual path may go through at most two marked slabs and traverses the middle part at most once. Hence, it is of length $\leq 2\cdot3n+8n=14n$. This finishes the proof.
\end{proof}

The parity restriction is only technical we can double any red line or remove any blue line while keeping all crucial properties. 

The idea of the above example cannot be improved. We can neither add more blue lines nor remove any red line without dramatically increase the length of the longest  dual path. In particular we will see later that at least one face with many lines of the other color traversing through it is necessary for a linear bound on the length of the longest alternating dual  path. Therefore we ask the following tantalizing open questions:

\section{Random Coloring}\label{sec:ran}
In this section we consider the following question:
Does every pseudoline arrangement have a coloring of the pseudolines that admits a
'long' alternating path?
We answer this question affirmatively by Theorem~\ref{thm:probLongPath}.


\RC*

This implies that each pseudoline arrangement has a coloring that admits an
alternating path of this length.


For the proof of Theorem~\ref{thm:probLongPath} we define an 
orientation of the dual graph of 
the arrangement as in~\cite{cellpaths}:
As $G^*(A)$ of the arrangement is bipartite, we
fix a coloring of the faces with black and white,
such that no two adjacent faces obtain the same color.
As shown in the left of Figure~\ref{fig:random}, we direct an edge from a white to
a black vertex if they share a blue edge.
If they share a red edge we direct an edge from the black to the white face. 
We denote this graph by $\overrightarrow{G}^*(\mathcal{A})$.
A directed path in $\overrightarrow{G}^*(\mathcal{A})$ corresponds
to an alternating path in the arrangement.
Note that two intersecting alternating paths can be glued together.
The set of faces that can be reached by a directed path from the face $z$
is denoted by $\reach(z)$.

The proof of the lemma above is based on the fact that the faces incident to a
bichromatic vertex induce a cycle in $\overrightarrow{G}^*(\mathcal{A})$.

\begin{lemma}[\cite{cellpaths}]\label{lem:monochromaticBorder}
    Let $E$ be the set of edges of the arrangement that separate faces
    in $\reach(z)$ from its complement.
    No red edge in $E$ shares a vertex with a blue edge in $E$.
\end{lemma}
%
%
%
We define the \emph{depth} of a face as the
minimum number of lines separating a face from an unbounded one.
The \emph{outer tunnel} $\mathcal{O}^w$ 
is the set of faces of depth at most $w$.

The idea is similar to the uncolored case: We find a path in 
tunnel $T_i^w$ from left to right or from right to left. 
Those paths are glued together in the outer tunnel.
The outer tunnel $\mathcal{O}$ is needed 
to glue those paths together to one.

We construct a path in the following way:
Assume there exists directed paths $L$ and $R$ 
in the outer tunnel from some face in $\mathcal{T}^w_0$ 
to $\mathcal{T}^w_l$ as in the middle of Figure~\ref{fig:random}.
The path $L$ crosses the wall of $\mathcal{T}^w_1, \mathcal{T}^w_2 , \, \ldots$.
There are some faces $F$ traversed immediately before and immediately after $L$ crosses a wall.
Denote these faces by $\tilde l_1,\tilde l_2, \, \ldots$ in the order they are traversed by $L$.
Similarly, we define $\tilde r_1,\tilde r_2,\, \ldots$ for the path $R$. 
Note that each $\tilde l_i,\tilde r_i$ is adjacent to the top wall of a tunnel for $i$ odd and 
to the bottom wall for $i$ even.

We define paths $P_i$, if they exist, as follows
\[P_i:=\text {a path in tunnel } \mathcal{T}_i\ \begin{cases}
       \text{from } \tilde l_{2i} \text{ to } \tilde r_{2i+1} & i \text{ even,}\\
       \text{from } \tilde r_{2i} \text{ to } \tilde l_{2i+1}& i \text{ odd.}
      \end{cases}\]
If there is more than one such path we pick an arbitrary one.
Our long path is $P = P_1P_2P_3 \ldots $ as depicted in the middle of Figure~\ref{fig:random}.

It remains to show that 
\begin{enumerate}[itemsep=0pt]
	\item[(i)]\label{item:1} 
	the paths $P_i$ exists with high probability,
	\item[(ii)]\label{item:2} 
	the paths $L$ and  $R$ exist with high probability, and
	\item[(iii)]\label{item:3} 
	the path $P$ has length $\Omega (n^2/\log n)$.
\end{enumerate} 
We first show~(i).
\begin{lemma}\label{lem:tunnelBlocked}
 Assume $L$ and $R$ exist as described above.
 Let $A_i$ be the event that $P_i$ does not exist.   
 Then \[ Pr(A_i)<n^4\ 2^{-w+3}.\]
\end{lemma}

\begin{proof}
  Assume $i$ is even and there exists no path from $\tilde l_{2i}$ to $\tilde r_{2i+1}$.
  Then consider the boundary of the reachable region $\reach(\tilde l_{2i})$.
  Since there exists a path from $\tilde l_{2i}$ to $\tilde l_{2i+1}$,
  the boundary of $\reach(\tilde l_2i)$ connects the top and bottom wall.
  Let $u$ and $v$ be the points where the boundary 
  touches the upper and lower wall furthest to the right.
  By Lemma~\ref{lem:monochromaticBorder}, the vertices on the boundary of the
  reachable region between $u$ and $v$ are of the same color, see
  Figure~\ref{fig:random} right.
  \begin{figure}[h]%
	\centering
	\includegraphics{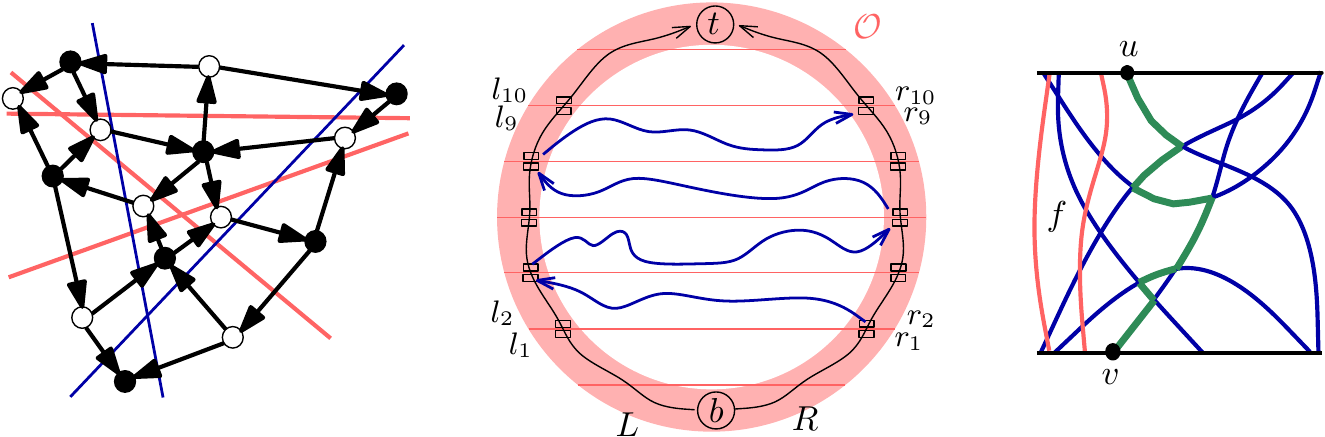}%
	\caption{Left: the definition of $\protect\overrightarrow{G}^*(\mathcal{A})$.
  Middle: construction of the path by tunnel paths.
  Right: the boundary of 
	$reach(f)$ within the tunnel is marked green and is monochromatic.}
	\label{fig:random}%
	\end{figure}
        
  As $u$ and $v$ are adjacent to faces that are $w-2$ levels apart, the set $S_{u,v}$ of lines
  separating $u$ from $v$ contains at least $w-2$ lines.
  All these lines intersect the boundary between $u$ and $v$,
  and thus all lines are of the same color.
        
  The probability that for two fixed points $x, y$ the 
  set $S_{x,y}$ is monochromatic is upper bounded by ${2}/(2^{w-2})$.

  The complexity of a wall is bounded by $O(n^{2})$,
  thus there are at most $C n^{4}$ pairs of points 
  on walls for some constant $C$. 
  So the probability that a path within the wall is blocked is bounded by
 \begin{equation*}
 Pr(A_i) \leq \sum_{(x,y)}Pr(S_{x,y}\mbox{ is monochromatic})\leq \frac{Cn^{4}}{2^{w-3}}. \qedhere
 \end{equation*}
\end{proof}

A similar proof also works to show that with high
probability the paths $L$ and $R$ exist.

\begin{lemma}\label{lem:circle}
   Let $B$ be the event that one of the paths $L$ or $R$ does not exist.
   Then \[\Prob(B)<n^3\frac{1}{2^{w-1}}+\frac{2}{2^{n/3-w}}.\]
\end{lemma}        

\begin{proof}
   First we show that there is an alternating path from an unbounded face
   to the faces of $\mbox{faces}(\mathcal{A}) \setminus \mbox{faces}(\mathcal{O})$
   with high probability if $w=o(n)$.
   This is due to the fact, that the directed graph defined on the cells is
   connected~\cite{cellpaths},
   and $n/3-w$ lines intersect $\mathcal{A}-\mathcal{O}$, which follows from the fact that there
   is a face of depth $n/3$ in each line arrangement, see~\cite{rousseeuw1999depth}.
   The probability that all these lines have the same color is bounded by
   $\frac{2}{2^{n/3-w}}$.
   
   The rest of the argument follows the lines of the proof of Lemma~\ref{lem:tunnelBlocked}:
   We have an alternating path touching both boundaries of our region.
   A blocking boundary of the reachable region to one side (clockwise/counterclockwise)
   produces a certificate of a set of $w-1$ lines, which have to be monochromatic,
   which only happens with low probability. The number of possible endpoints
   of the boundary on is $2n$, on the inner $w$-layer circle at most ${n \choose 2}$.
   So the probability that an outer alternating cycle does not exist is at most
   $$\Prob(B)\leq\frac{2}{2^{n/3-w}}+n^3\cdot\frac{1}{2^{w-1}}.$$
   Note that the lack of such a certificate of $w-1$ monochromatic lines implies the existence of a
   clockwise and an anticlockwise alternating cycle.
\end{proof}

\noindent{\bfseries Proof of Theorem~\ref{thm:probLongPath}.}
We choose $w= 6\lceil\log n\rceil+3$.
Using the union bound, we can bound the probability that at least one of the
paths $P_i$ does not exist from above by
\begin{align*}
 Pr\left(\bigvee A_i \vee B\right) \leq \sum Pr(A_i)+Pr(B) \leq \sum_{i=1}^{{n}/{w}}\frac{Cn^4}{2^{6\log n}}+Pr(B) = o(1).
\end{align*}
%
%
It remains to show (iii), namely to give a lower bound on the length of $P$.
For this, we extend $P_i$ by adding a shortest path from the start vertex (and
the end vertex) to a left (and right) unbounded face within the tunnel $\mathcal T_i$.
The length of the two extensions can be upper bounded by $2w$ each: at most $w$ faces
are needed to reach any unbounded face and another $w$ to ensure an unbounded face
within tunnel $\mathcal T_i$. 
By Claim~\ref{clm:length}, each path $P_i$ in one of the middle tunnels, i.e., with $l/3\leq i\leq 2l/3$,
contains a linear number of faces:
$$|P_i|+4w\geq 2i(w-1)\geq 2\frac{n}{3w}\frac{w}{2}\geq n/3.$$
Moreover, the number of these paths is at least
$l/3 = \Omega(n/ \log n)$. Consequently, $P$ is of length $\Omega(n^2/\log n)$
and exists with high probability.
This finishes the proof.

\section{Open Problems}
To end, we want to state some interesting open problems:
\begin{itemize}
 \item Find a tight bound for the bicolored case: 
 What is the minimal length of a longest alternating path over all bicolored arrangements with $n$ lines? (The answer lies between $n$ and $2n+1$.)
 \item Does every pseudoline arrangement with $n$ red and $n$
 blue lines have an alternating dual path of  length $\Omega (n^2)$?
 \item Is there a coloring for every pseudoline arrangement such that there exist an alternating path of length $\Omega(n^2)$?
\end{itemize}

\paragraph*{Acknowledgment}

We thank Nieke Aerts, Stefan Felsner, Heuna Kim and Piotr Micek for interesting
and helpful discussions on the topic. Further, we want to thank anonymous reviewers 
for their helpful comments regarding the presentation.

\bibliographystyle{plain} 
\bibliography{Lib}

\end{document}